\documentclass[b5paper,reqno]{amsart}
\usepackage{amsmath,amssymb,amsfonts,amsthm,latexsym,amstext,amscd}
\usepackage{epsfig,graphics,graphicx,color}
\usepackage{epstopdf}
\usepackage{pdfsync}
\usepackage{enumerate}
\usepackage[applemac]{inputenc}
\usepackage{color}
\usepackage{hyperref}\hypersetup{pdfborder={0 0 0},colorlinks,citecolor=red,filecolor=red,linkcolor=red,urlcolor=green}
\makeatletter \thm@headfont{\bfseries\scshape} \makeatother
\allowdisplaybreaks

\newtheorem{thm}{Theorem}
\newtheorem{lem}[thm]{Lemma}
\newtheorem{cor}[thm]{Corollary}
\newtheorem{pro}[thm]{Proposition}

\newtheorem{rem}[thm]{Remark}
\def\SU{{\displaystyle \sum_1^\infty}_{\raise 6pt\hbox{$\scriptstyle u$}}}
\def\PIU{{\displaystyle \prod_1^\infty}_{\raise 6pt\hbox{$\scriptstyle u$}}}
\def\SUK{{\displaystyle \sum_1^k}_{\raise 6pt\hbox{$\scriptstyle u$}}}
\def\SUKZERO{{\displaystyle \sum_{k_0}^\infty}_{\raise 6pt\hbox{$\scriptstyle u$}}}
\def\SUKL{{\displaystyle \sum_k^\ell}_{\raise 6pt\hbox{$\scriptstyle u$}}}

\begin{document}\parindent0pt
\title[Almost sure convergence of products of $2\times2$ nonnegative matrices]
      {Almost sure convergence of products of $2\times2$ nonnegative matrices}


\author[A. Thomas]{Alain Thomas}
\address[Alain Thomas]{
LATP, 39, rue Joliot-Curie,\hfil\break
13453 Marseille, Cedex 13,
France}
\email{thomas@cmi.univ-mrs.fr}






\begin{abstract}
We study the almost sure convergence of the normalized columns in an infinite product of nonnegative matrices, and the almost sure rank one property of its limit points. Given a probability on the set of $2\times2$ nonnegative matrices, with finite support $\mathcal A=\{A(0),\dots,A(s-1)\}$, and assuming that at least one of the $A(k)$ is not diagonal, the normalized columns of the product matrix $P_n=A(\omega_1)\dots A(\omega_n)$ converge almost surely (for the product probability) with an exponential rate of convergence if and only if the Lyapunov exponents are almost surely distinct. If this condition is satisfied, given a nonnegative column vector $V$ the column vector $\frac{P_nV}{\Vert P_nV\Vert}$ also converges almost surely with an exponential rate of convergence. On the other hand if we assume only that at least one of the $A(k)$ do not have the form $\begin{pmatrix}a&0\\0&d\end{pmatrix}$, $ad\ne0$, nor the form $\begin{pmatrix}0&b\\d&0\end{pmatrix}$, $bc\ne0$, the limit-points of the normalized product matrix $\frac{P_n}{\Vert P_n\Vert}$ have almost surely rank~$1$ --~although the limits of the normalized columns can be distinct~-- and $\frac{P_nV}{\Vert P_nV\Vert}$ converges almost surely with a rate of convergence that can be exponential or not exponential.

\end{abstract}
%
\maketitle
\begin{centerline}%
{\dedicatory\textsl{}}
\end{centerline}

\centerline{\sc Introduction}

\bigskip

Given a finite set of nonnegative matrices
$$
A(0)=\begin{pmatrix}a(0)&b(0)\\c(0)&d(0)\end{pmatrix},\dots,A(s-1)=\begin{pmatrix}a(s-1)&b(s-1)\\c(s-1)&d(s-1)\end{pmatrix}
$$
we consider the product matrix
$$
P_n(\omega)=A(\omega_1)\dots A(\omega_n)=\begin{pmatrix}\alpha_n(\omega)&\beta_n(\omega)\\\gamma_n(\omega)&\delta_n(\omega)\end{pmatrix},\quad\omega=(\omega_n)\in\{0,1,\dots,s\}^\mathbb N
$$
and we are interested by the almost sure limit points of $\frac{P_n}{\Vert P_n\Vert}$ and the almost sure convergence of the normalized columns of $P_n$. The set $\{0,\dots,s-1\}^{\mathbb N}$ is endowed by the product probability, defined from $p_0,\dots,p_{s-1}>0$ with $\sum_kp_k=1$. This product is easily computable when all the matrices are upper triangular, or when all the matrices are lower triangular, and also when they are stochastic \cite[Proposition 1.2]{sto}.

The book of Bougerol and Lacroix, since its purpose is different, do not give any indication in the particular case that we study in this paper: indeed the hypothesis of contraction they make, for the almost sure convergence of $\frac{P_nV}{\Vert P_nV\Vert}$ (\cite[Part A III Theorem 4.3]{BL}), is the existence -- for any $k\in\mathbb N$ -- of a matrix~$M_k$, product of matrices belonging to the set $\{A(0),\dots,A(s-1)\}$, such that $\frac{M_k}{\Vert M_k\Vert}$ converges to a rank $1$ matrix when $k\to\infty$ (\cite[Part A III Definition 1.3]{BL}).

On the other hand, the weak ergodicity defined in \cite[Definition 3.3]{S} holds almost surely if and only if the product matrix $P_n$ is almost surely positive for $n$ large enough. However the strong ergodicity (\cite[Definition 3.4]{S}) do not necessary hold, even if all the $A(k)$ are positive.

The outset of the present study is the following theorem. The norm we use is the norm-$1$, and we say that the normalized columns of $P_n(\omega)$ converge if each column of $P_n(\omega)$ divided by its norm-$1$ (if nonnull) converges in the usual sense. We say that the rate of convergence is exponential or geometric when the difference -- between the entries of the normalized column and their limits -- is less than $Cr^n$ with $C>0$ and $0<r<1$.

\begin{thm}\label{main}
Let $\mathcal A=\{A(0),\dots,A(s-1)\}$ be a finite set of nonnegative matrices such that at least one of the $A(k)$ is not diagonal.

(i) The normalized columns of $P_n$ converge almost surely with an exponential rate of convergence if and only if  the singular values $\lambda_i(n)$ of $P_n$ satisfy almost surely
\begin{equation}\label{*}
\lim_{n\to\infty}(\lambda_1(n))^{\frac1n}\ne\lim_{n\to\infty}(\lambda_2(n))^{\frac1n}.
\end{equation}

(ii) If (\ref{*}) holds the limit-points of the normalized matrix $\frac{P_n}{\Vert P_n\Vert}$ have almost surely rank~$1$ and, given a nonnegative column vector $V$, the normalized column vector $\frac{P_nV}{\Vert P_nV\Vert}$ converges almost surely with an exponential rate of convergence.

(iii) Nevertheless the normalized matrix $\frac{P_n}{\Vert P_n\Vert}$ diverges almost surely, except in the case where the matrices $A(0),\dots,A(s-1)$ have a common left eigenvector.
\end{thm}


The different cases are detailed below:

$\_$ the case where at least one of the $A(k)$ has rank one, in Remark \ref{rem} (i),

$\_$ the cases where all the $A(k)$ have rank two and at least one of the $A(k)$ has more than two nonnull entries in Sections \ref{positive} and \ref{triangular},

$\_$ the cases where all the $A(k)$ have rank two and two nonull entries in Section \ref{diagonal}.


\begin{rem}\label{rem}
(i) If one of the matrices $A(k)$ has rank $1$, the normalized columns of $P_n$ are almost surely constant and equal for $n$ large enough, as well as $\frac{P_nV}{\Vert P_nV\Vert}$ for any nonnegative column vector $V$ such that $\forall n,P_nV\ne0$. So we can suppose in the sequel that all the $A(k)$ have rank $2$.

(ii) If both normalized columns of $P_n$ converge (resp. converge exponentially) to the same limit, then for any nonnegative column vector $V$ the normalized column $\frac{P_nV}{\Vert P_nV\Vert}$ is a nonnegative linear combination of them, so it converges (resp. it converges exponentially) to the same limit. The limit points of $\frac{P_n}{\Vert P_n\Vert}$ have rank~$1$.

(iii) If $\frac{P_n}{\Vert P_n\Vert}$ converges (resp. converges exponentially), then for any nonnegative column vector $V$ the normalized column $\frac{P_nV}{\Vert P_nV\Vert}$ converges (resp. converges exponentially) because it is $\frac{\frac{P_n}{\Vert P_n\Vert}V}{\Vert\frac{P_n}{\Vert P_n\Vert}V\Vert}$.

\end{rem}

\section{Some triangular examples}

$\_$ A case where the normalized columns converge almost surely to the same limit with a convergence rate in $\frac1n$. Suppose that $A(k)=\begin{pmatrix}1&b(k)\\0&1\end{pmatrix}$ and that $\exists k_0,\ b(k_0)\ne0$. We have
$$
P_n=\begin{pmatrix}1&\sum_{i=1}^nb(\omega_i)\\0&1\end{pmatrix}.
$$
We obtain the normalized second column (resp. the normalized matrix) by dividing the second column of $P_n$ (resp. $P_n$ itself) by $1+\sum_{i=1}^nb(\omega_i)$. So both normalized columns converge almost surely to $\begin{pmatrix}1\\0\end{pmatrix}$, and the normalized matrix converges almost surely to $\begin{pmatrix}0&1\\0&0\end{pmatrix}$. Since the density of the set $\{i\;;\;\omega_i=k_0\}$ is alomst surely $\frac1s$, $\sum_{i=1}^nb(\omega_i)$ has the order of growth of $n$. So the convergence rate of the second normalized column, and the one of the normalized matrix $\frac{P_n}{\Vert P_n\Vert}$, have the order of~$\frac1n$.

$\_$ A case where the normalized columns converge almost surely to the same limit with an exponential convergence rate. Suppose that $A(k)=\begin{pmatrix}2&b(k)\\0&1\end{pmatrix}$ and that $\exists k_0,\ b(k_0)\ne0$. We have
$$
P_n=\begin{pmatrix}2^n&\sum_{i=1}^n2^{i-1}b(\omega_i)\\0&1\end{pmatrix}
$$
so the normalized columns of $P_n$ converge almost surely to $\begin{pmatrix}1\\0\end{pmatrix}$ and the limit points of the normalized matrix $\frac{P_n}{\Vert P_n\Vert}$ have the form $\begin{pmatrix}\alpha&1\\0&0\end{pmatrix}$ or $\begin{pmatrix}1&\alpha\\0&0\end{pmatrix}$, $\alpha\in[0,1]$.

$\_$ A case where the normalized columns converge almost surely to two different limit with an exponential convergence rate. Suppose that $A(k)=\begin{pmatrix}1&b(k)\\0&2\end{pmatrix}$, we have
$$
P_n=\begin{pmatrix}1&\sum_{i=1}^n2^{n-i}b(\omega_i)\\0&2^n\end{pmatrix}.
$$
so the limit of the second normalized column is $\begin{pmatrix}\frac s{s+1}\\\frac1{s+1}\end{pmatrix}$ with $s=\sum_{i=1}^\infty2^{-i}b(\omega_i)$.

$\_$ Another case where the normalized columns converge almost surely to the same limit, but the convergence rate is non-exponential. Suppose that the alphabet has two elements, $A(0)=\begin{pmatrix}2&2\\0&1\end{pmatrix}$ and $A(1)=\begin{pmatrix}1&1\\0&2\end{pmatrix}$, and that $p_0=p_1=\frac12$. We have
$$
P_n=\begin{pmatrix}2^{k_0(n)}&2^{k_1(n)}\sum_{i=1}^n2^{k_0(i)-k_1(i)}\\0&2^{k_1(n)}\end{pmatrix}\hbox{ where }k_j(i)=\#\{i'\le i\;;\;\omega_{i'}=j\}.
$$
By the well known recurrence property one has almost surely $k_0(i)=k_1(i)$ for infinitely many~$i$, so both normalized columns converge to $\begin{pmatrix}1\\0\end{pmatrix}$. The difference between this vector and the second column of $P_n$ has entries $\pm\frac1{1+\sum_{i=1}^n2^{k_0(i)-k_1(i)}}$. With probability $1$, this difference do not converge exponentially to $0$ because \hbox{$\lim_{i\to\infty}\frac{k_0(i)-k_1(i)}i=0$}.

\section{The singular values of $P_n$}

The singular values of $P_n=\begin{pmatrix}\alpha_n&\beta_n\\\gamma_n&\delta_n\end{pmatrix}$, let $\lambda_1(n)$ and $\lambda_2(n)$, are by definition the positive roots of the eigenvalues of $^tP_nP_n$:
$$
\hskip-23pt\lambda_i(n):=\sqrt{\frac{{\alpha_n}^2+{\beta_n}^2+{\gamma_n}^2+{\delta_n}^2\pm{\sqrt{({\alpha_n}^2+{\beta_n}^2+{\gamma_n}^2+{\delta_n}^2)^2-4(\alpha_n\delta_n-\beta_n\gamma_n)^2}}}2}.
$$
Now the Lyapunov exponents $\lambda_i:=\lim_{n\to\infty}\frac1n\log\lambda_i(n)$ exist almost surely by the subadditive ergodic theorem \cite{R}, and one has
$$
\lambda_1=\lim_{n\to\infty}\frac1{2n}\log({\alpha_n}^2+{\beta_n}^2+{\gamma_n}^2+{\delta_n}^2)$$\begin{equation}\label{**}\lambda_2=\lambda_1+\lim_{n\to\infty}\frac1n\log\frac{\vert\alpha_n\delta_n-\beta_n\gamma_n\vert}{{\alpha_n}^2+{\beta_n}^2+{\gamma_n}^2+{\delta_n}^2}.
\end{equation}
Notice that $\lambda_1$ is finite if $P_n$ is not eventually the null matrix: denoting by $\alpha$ and $\beta$ the smaller nonnull value and the greater value of the entries of the matrices $A(k)$ one has
$$
\log\alpha\le\lambda_1\le\log(2\beta).
$$
As for $\lambda_2$, it belongs to $[-\infty,\lambda_1]$. For instance if $P_{n_0}$ has rank $1$, $\lambda_2=\lambda_2(n)=-\infty$ for $n\ge n_0$.

For any nonnegative matrix $M=\begin{pmatrix}a&b\\c&d\end{pmatrix}$ we denote by $d_{\mathcal H}(M)$ the Hilbert distance between the columns of $M$ and by $d_\infty(M)$ the norm-infinite distance between the normalized columns of $M$:
$$
d_{\mathcal H}(M):=\left\vert\log\frac{ad}{bc}\right\vert\quad\hbox{and}\quad d_\infty(M)=\left\vert\frac a{a+c}-\frac b{b+d}\right\vert.
$$

\begin{pro}\label{lower}
$d_{\mathcal H}(M)$ and $d_\infty(M)$ are at least equal to $\frac{\vert ad-bc\vert}{a^2+b^2+c^2+d^2}$.
\end{pro}

\begin{proof}
By the classical inequality $\vert\log t\vert\ge\frac{2\vert t-1\vert}{t+1}$ one has $d_{\mathcal H}(M)\ge\frac{2\vert ad-bc\vert}{ad+bc}$. Now $ad\le a^2+d^2$ and $bc\le b^2+c^2$, so $d_{\mathcal H}(M)\ge\frac{\vert ad-bc\vert}{a^2+b^2+c^2+d^2}$.

On the other side $d_\infty(M)=\frac{\vert ad-bc\vert}{ab+ad+bc+cd}\ge\frac{\vert ad-bc\vert}{a^2+b^2+c^2+d^2}$ because
$$
2(a^2+b^2+c^2+d^2)=(a^2+b^2)+(a^2+d^2)+(b^2+c^2)+(c^2+d^2)\ge2ab+2ad+2bc+2cd.
$$\end{proof}

\begin{cor}\label{exp}
If $d_{\mathcal H}(P_n)$ or $d_\infty(P_n)$ converges exponentially to $0$, the normalized columns of $P_n$ converge exponentially to the same limit and $\lambda_2<\lambda_1$.
\end{cor}

\begin{proof}Let $p,q\ge n$. The Hilbert distance between the $i^{\rm th}$ column of $P_p$ and the $j^{\rm th}$ column of $P_q$ is $d_{\mathcal H}(P_nP_{n,p,q})$, where $P_{n,p,q}$ is the matrix whose columns are the $i^{\rm th}$ column of $A(\omega_{n+1})\dots A(\omega_p)$ and  the $j^{\rm th}$ column of $A(\omega_{n+1})\dots A(\omega_q)$. By the well known property of the Birkhoff coefficient $\tau_{\mathcal B}$ (\cite[Section 3]{S}) this distance is at most $d_{\mathcal H}(P_n)\tau_{\mathcal B}(P_{n,p,q})\le d_{\mathcal H}(P_n)$, so it converge exponentially to $0$.

By the obvious inequality $d_\infty(M)\le d_{\mathcal H}(M)$ the norm-infinite distance between the $i^{\rm th}$ normalized column of $P_p$ and the $j^{\rm th}$ normalized column of $P_q$ ($i=j$ or $i\ne j$) converges exponentially to $0$, so both normalized columns are Cauchy and converge exponentially to the same limit when $n$ tends to infinity, and $\lambda_2<\lambda_1$ by (\ref{**}) and Proposition \ref{lower}.\end{proof}

\section{The case where $P_n$ is almost surely positive for $n$ large enough}\label{positive}

\begin{thm}
Suppose that $A^*=\sum_kA(k)$ is not triangular and that at least one of the $A(k)$ has more than two nonnull entries. Then, with probability $1$, the normalized columns of $P_n$ converge exponentially to the same limit, as well as $\frac{P_nV}{\Vert P_nV\Vert}$ for any nonnegative column vector $V$, and $\lambda_2<\lambda_1$; the limit points of $\frac{P_n}{\Vert P_n\Vert}$ have rank $1$; the weak ergogdicity, in the sense of \cite[Definition 3.3]{S}, holds.
\end{thm}

\begin{proof}
By the hypotheses either one of the $A(k)$ is positive, or one of the $A(k)$ is $\begin{pmatrix}a(k)&b(k)\\c(k)&0\end{pmatrix}$ (its square is positive), or  one of the $A(k)$ is triangular with three nonnull entries and another $A(k')$ is such that $A(k)+A(k')>0$. In this last case $A(k)A(k')A(k)$ is positive. So in all cases there exist $k,k'$ such that $A(k)A(k')A(k)$ is positive and, denoting by $k(n)$ the number of occurences of the word $kk'k$ in $\omega_1\dots\omega_n$, the limit of $\frac{k(n)}n$ is almost surely $s^{-3}$.

Clearly the Hilbert distance $d_{\mathcal H}$ and the Birkhoff coefficient $\tau_{\mathcal B}$ (\cite[Section 3]{S}) have the following property:
$$
d_{\mathcal H}(M_1\dots M_i)\le d_{\mathcal H}(M_1)\tau_{\mathcal B}(M_2)\dots\tau_{\mathcal B}(M_i)
$$
where $d_{\mathcal H}$ means the distance between the rows (or the columns) of $M$.

Now we split the product matrix $P_n$ in the following way:
$$
P_n=M_1\dots M_k\hbox{ where }M_i=A(\omega_{n_{i-1}+1})\dots A(\omega_{n_i}),\ n_0=0<n_1<\dots<n_k,
$$
the indexes $n_1,n_3,\dots$ corresponding to the disjoint occurrences of $A(k)A(k')A(k)$:
$$
M_1=A(\omega_1)\dots A(\omega_{n_1-3})A(k)A(k')A(k)\hbox{ and }M_3=M_5=\dots=A(k)A(k')A(k).
$$
Let $C=d_{\mathcal H}(M_1)<\infty$, $r=\tau_{\mathcal B}(A(k)A(k')A(k))<1$ and $r'\in]r^{s^{-3}/3},1[$. We have for $n$ large enough
$$
d_{\mathcal H}(P_n)\le C\tau(A(\omega_{n_1+1})\dots A(\omega_n))\le Cr^{k(n)/3-1}\le Cr'^n
$$
and we conclude with Corollary \ref{exp} and \cite[Lemma 3.3]{S}.\end{proof}

\section{The case where $A^*=\sum_kA(k)$ is triangular not diagonal}\label{triangular}

Assuming for instance that $A^*$ is upper triangular not diagonal, we have
$$
P_n(\omega)=\begin{pmatrix}\alpha_n(\omega)&\delta_n(\omega)s_n(\omega)\\0&\delta_n(\omega)\end{pmatrix}\quad\hbox{with}\quad s_n(\omega)=\sum_{i=1}^n\frac{\alpha_{i-1}(\omega)}{\delta_{i-1}(\omega)}\frac{b(\omega_i)}{d(\omega_i)}.
$$
To know if $\lim_{n\to\infty}s_n(\omega)$ is finite or infinite, and to know the rate of convergence, we use the exponentials of the expected values of $\log a(\cdot)$ and $\log d(\cdot)$:
\begin{equation}\label{***}
p:=a(0)^{p_0}\dots a(s-1)^{p_{s-1}}\hbox{ and }q:=d(0)^{p_0}\dots d(s-1)^{p_{s-1}}.
\end{equation}
By the law of large numbers, for any $\varepsilon>0$ we have almost surely for any integer $n\ge0$
\begin{equation}\label{****}
\kappa\ \frac{p^n}{q^n}\ (1-\varepsilon)^n\le\frac{\alpha_n(\omega)}{\delta_n(\omega)}\le K\ \frac{p^n}{q^n}\ (1+\varepsilon)^n\qquad(\kappa,K\hbox{ constants}).
\end{equation}
On the other side, since $A^*$ is not diagonal, the set of the integers $n$ such that $b(\omega_n)\ne0$ has almost surely the positive density $\frac1s\#\{k\;;\;b(k)\ne0\}$. Given $\varepsilon>0$, for $n$ large enough this set has a nonempty intersection with $[n(1-\varepsilon),n]$. So we deduce from (\ref{****}) that -- replacing eventually $\kappa$ by a smallest constant and $K$ by a greater constant
\begin{equation}\label{*****}
\kappa\ \frac{p^n}{q^n}\ (1-\varepsilon)^n\le s_n(\omega)\quad\hbox{and, if }p\ge q,\quad s_n(\omega)\le K\ \frac{p^n}{q^n}\ (1+\varepsilon)^n.
\end{equation}

\begin{thm}\label{tri}
Suppose that $A^*$ is upper triangular not diagonal, and that all the $A(k)$ have rank $2$. Then $P_n$ has almost surely the following properties:

(i) If $p\ge q$, $\frac{P_nV}{\Vert P_nV\Vert}$ converges to $\begin{pmatrix}1\\0\end{pmatrix}$ for any nonnegative column vector $V$, and the limit points of $\frac{P_n}{\Vert P_n\Vert}$ have rank $1$. The rate of convergence of $\frac{P_nV}{\Vert P_nV\Vert}$ is exponential if and only if $p>q$, and one has $\lambda_2<\lambda_1$ if and only if $p>q$.

(ii) If $p<q$, setting $s(\omega)=\lim_{n\to\infty}s_n(\omega)$, $\frac{P_nV}{\Vert P_nV\Vert}$ converges exponentially to $\begin{pmatrix}\frac{s(\omega)}{s(\omega)+1}\\\frac1{s(\omega)+1}\end{pmatrix}$ if and only if the second entry of $V$ is nonnull. The normalized product matrix $\frac{P_n}{\Vert P_n\Vert}$ converges exponentially to $\begin{pmatrix}0&\frac{s(\omega)}{s(\omega)+1}\\0&\frac1{s(\omega)+1}\end{pmatrix}$ and $\lambda_2<\lambda_1$.
\end{thm}

\begin{proof}(i) If $p\ge q$ one has almost surely $\lim_{n\to\infty}s_n(\omega)=\infty$ because, among the indexes $i$ such that $b_i(\omega)\ne0$, one has $\frac{\alpha_{i-1}(\omega)}{\delta_{i-1}(\omega)}\ge1$ infinitely many times. The difference between $\begin{pmatrix}1\\0\end{pmatrix}$ and the second normalized column of $P_n$ is $\frac1{s_n(\omega)+1}\begin{pmatrix}1\\-1\end{pmatrix}$, so it tends almost surely to $0$ by (\ref{*****}), with an exponential rate of convergence if and only if $p>q$. We conclude about $\frac{P_nV}{\Vert P_nV\Vert}$ and $\frac{P_n}{\Vert P_n\Vert}$, by using Remark~\ref{rem} (ii).

If $p>q$ one has almost surely, from (\ref{**}) and (\ref{****}), $\lambda_2\le\lambda_1+\lim_{n\to\infty}\frac1n\log\frac{\delta_n(\omega)}{\alpha_n(\omega)}<\lambda_1$. If $p=q$, (\ref{**}), (\ref{****}) and (\ref{*****}) imply almost surely $\lambda_1=\lambda_2$.

(ii) If $p<q$, the almost sure exponential convergence of $\frac{P_nV}{\Vert P_nV\Vert}$ and $\frac{P_n}{\Vert P_n\Vert}$ is due to the fact that $s(\omega)-s_n(\omega)$ tends exponentially to $0$ from (\ref{****}). The almost sure inequality $\lambda_2<\lambda_1$ is due to $\lambda_2\le\lambda_1+\lim_{n\to\infty}\frac1n\log\frac{\alpha_n(\omega)}{\delta_n(\omega)}$.\hfill\end{proof} 

\section{The case where all the $A(k)$ have only two nonnull entries}\label{diagonal}

We assume that the $A(k)$ have rank $2$, so the normalized columns of $P_n$ are $\begin{pmatrix}1\\0\end{pmatrix}$ and $\begin{pmatrix}0\\1\end{pmatrix}$.

Suppose first that all the $A(k)$ are diagonal with two nonnul entries. If the constants $p$ and $q$ defined in (\ref{***}) are distinct, then with probability $1$ one has $\lambda_2<\lambda_1$, the normalized column vector $\frac{P_nV}{\Vert P_nV\Vert}$ converges exponentially for any nonnegative vector $V$, as well as the normalized matrix $\frac{P_n}{\Vert P_n\Vert}$. If $p=q$, then with probability $1$ we have $\lambda_1=\lambda_2$, the normalized column matrix $\frac{P_nV}{\Vert P_nV\Vert}$ do not necessarily converge, and the limit points of the normalized matrix $\frac{P_n}{\Vert P_n\Vert}$ have the form $\begin{pmatrix}\alpha&0\\0&1\end{pmatrix}$ or $\begin{pmatrix}1&0\\0&\alpha\end{pmatrix}$, $\alpha\in[0,1]$.

Suppose now that at least one of the $A(k)$ has the form $\begin{pmatrix}0&b(k)\\c(k)&0\end{pmatrix}$, $b(k)c(k)\ne0$. Let $i_1,i_2,\dots$ be the indexes such that $A(\omega_i)$ has this form. For any $n\in\mathbb N$ the product matrix $P_n$ or $P_n\begin{pmatrix}0&1\\1&0\end{pmatrix}$ has the diagonal form $\begin{pmatrix}\alpha_n&0\\0&\delta_n\end{pmatrix}$ with
$$
\alpha_n=a(\omega_1)\dots a(\omega_{i_1-1})b(\omega_{i_1})d(\omega_{i_1+1})\dots d(\omega_{i_2-1})c(\omega_{i_2})\dots
$$
$$
\delta_n=d(\omega_1)\dots d(\omega_{i_1-1})c(\omega_{i_1})a(\omega_{i_1+1})\dots a(\omega_{i_2-1})b(\omega_{i_2})\dots.
$$
Clearly, the first as well as the second normalized column diverge. With probability~$1$ the limit points of the normalized matrix $\frac{P_n}{\Vert P_n\Vert}$ have the form $\begin{pmatrix}\alpha&0\\0&1\end{pmatrix}$, $\begin{pmatrix}0&\alpha\\1&0\end{pmatrix}$, $\begin{pmatrix}0&1\\\alpha&0\end{pmatrix}$ or $\begin{pmatrix}1&0\\0&\alpha\end{pmatrix}$, $\alpha\in[0,1]$, and $\lambda_1=\lambda_2$.

\section{The almost sure divergence of $\frac{P_n}{\Vert P_n\Vert}$}

It remains to prove the last item of Theorem \ref{main}, and more generally the following

\begin{pro}\label{infinite}(i) Let $(A_n)$ be a sequence of complex-valued $d\times d$ matrices. If the sequence of normalized matrices $\frac{A_1\dots A_n}{\Vert A_1\dots A_n\Vert}$ converges, the matrices $A$ such that $A_n=A$ for infinitely many $n$ have a common left-eigenvector, which is a row of the limit matrix.

(ii) Given a set $\mathcal A=\{A(0),\dots,A(s-1)\}$ of complex-valued $d\times d$ matrices without common left-eigenvector, a positive probability vector $(p_0,\dots,p_{s-1})$ and the product probability on $\{0,\dots,s-1\}^{\mathbb N}$, the sequence of the normalized matrices $\frac{A(\omega_1)\dots A(\omega_n)}{\Vert A(\omega_1)\dots A(\omega_n)\Vert}$ diverges for almost all sequence $(\omega_n)_{n\in\mathbb N}\in\{0,\dots,s-1\}^{\mathbb N}$.
\end{pro}

\begin{proof}(i) Let $P_n=A_1\dots A_n$, we suppose that the limit $P=\lim_{n\to\infty}\frac{P_n}{\Vert P_n\Vert}$ exists and we denote $\lambda_n=\frac{\Vert P_n\Vert}{\Vert P_{n-1}\Vert}$. If $A_n=A$ for $n=n_1,n_2,\dots$ with $n_1<n_2<\dots$ one has
\begin{equation}\label{PA}
\begin{array}{rcl}PA&=&\lim_{k\to\infty}\frac{P_{n_k-1}}{\Vert P_{n_k-1}\Vert}A\\&=&\lim_{k\to\infty}\lambda_{n_k}\frac{P_{n_k}}{\Vert P_{n_k}\Vert}.\end{array}
\end{equation}
One deduce $\Vert PA\Vert=\lim_{k\to\infty}\left\Vert\lambda_{n_k}\frac{P_{n_k}}{\Vert P_{n_k}\Vert}\right\Vert=\lim_{k\to\infty}\lambda_{n_k}$, and the equality (\ref{PA}) becomes $PA=\Vert PA\Vert P$. Since the matrix $P$, of norm $1$, has at least one row with a nonnull entry, this row is a left-eigenvector of $A$ related to the eigenvalue $\Vert PA\Vert$.



(ii) Given $A\in\mathcal A$ and $n_0\in\mathbb N$, the set of the sequences $(\omega_n)$ such that $A(\omega_n)\in\mathcal A\setminus\{A\}$ for any $n\ge n_0$ has probability $0$. Hence, with probability $1$, the sequence $(A(\omega_n))$ has infinitely many occurrences of each of the matrices of $\mathcal A$ and, by (i), $\frac{A(\omega_1)\dots A(\omega_n)}{\Vert A(\omega_1)\dots A(\omega_n)\Vert}$ diverges almost surely.\hfill\end{proof}

\section{The rank one property of the infinite products of matrices}

Here is a general result, deduced from \cite{o}.

\begin{thm}\label{r1}Let $(A_n)$ be a sequence of complex-valued $d\times d$ matrices, we denote by $P_n$ their normalized product with respect to the euclidean norm:
$$
P_n:=\frac{A_1\dots A_n}{\Vert A_1\dots A_n\Vert_2}.
$$
There exists a sequence $(Q_n)$ of matrices of rank $1$ such that
$$
\lim_{n\to\infty}\Vert P_n-Q_n\Vert_2=0
$$
if and only if $i_1(n)=1$ for $n$ large enough and $\lim_{n\to\infty}\frac{\lambda_2(n)}{\lambda_1(n)}=0$, where the $i_j(n)$ and the $\lambda_j(n)$ arise from the singular values decomposition:
$$
P_n=U\begin{pmatrix}\lambda_1(n)\mathbb I_{i_1(n)}&\dots&0\\\vdots&\ddots&\vdots\\0&\dots&\lambda_\delta(n)\mathbb I_{i_\delta(n)}\end{pmatrix}V\quad\hbox{with }\lambda_1(n)>\lambda_2(n)>\dots>\lambda_\delta(n).
$$
\end{thm}

\begin{proof}Denoting by $C_i$ the columns of $U$, by $R_i$ the rows of $V$, and denoting by $s_j(n)$ the sum $i_1(n)+\dots+i_j(n)$ we have 
\begin{equation}\label{svd}
P_n=\sum_{0<i\le s_1(n)}C_i(n)R_i(n)+\frac{\lambda_2(n)}{\lambda_1(n)}\sum_{s_1(n)<i\le s_2(n)}C_i(n)R_i(n)+\dots
\end{equation}
so the converse implication of the theorem holds with $Q_n:=C_1(n)R_1(n)$.

To prove the direct implication we need the following lemma:

\begin{lem}\label{rank}Any matrix $A$ of the form $A=\sum_{i=1}^rC_iR_i$, where the nonnull columns $C_i$ are orthogonal as well as the nonnull rows $R_i$, has rank $r$.
\end{lem}

\begin{proof}We complete $\{R_1,\dots,R_r\}$ to a orthogonal base. The rank of $A$ is the rank of the family $\{A{R_i}^*\}$, but $A{R_i}^*=\left\{\begin{array}{ll}C_i&(i\le r)\\0&(i>r)\end{array}\right.$ and consequently this family has rank $r$.\hfill\end{proof}

Now using the compacity of the set of vectors of norm $1$, there exists at least one increasing sequence of integers $(n_k)$ such that the columns $C_i(n_k)$, the rows $R_i(n_k)$, the reals $\frac{\lambda_i(n)}{\lambda_1(n)}$ and the integers $s_j(n_k)$ converge. Let $C_i$, $R_i$, $\alpha_i$ and $s_j$ be their respective limits, one deduce from (\ref{svd}) that $(P_{n_k})$ converge and
$$
\lim_{k\to\infty}P_{n_k}=\sum_{0<i\le s_1(n)}C_iR_i+\alpha_2\sum_{s_1<i\le s_2}C_iR_i+\dots.
$$
If $i_1(n)\ne1$ for infinitely many $n$, or if $\frac{\lambda_2(n)}{\lambda_1(n)}$ do not converge to $0$, we can choose the sequence $(n_k)$ such that $i_1\ge2$ or $\alpha_2\ne0$. Consequently -- from Lemma \ref{rank} -- $\lim_{k\to\infty}P_{n_k}$ has rank at least $2$, so it is not possible that $\lim_{n\to\infty}\Vert P_n-Q_n\Vert_2=0$ with $Q_n$ of rank $1$.\hfill\end{proof}

{\bf Acknowledgments.} We are grateful to Pr. Ludwig Elsner and associate mathematicians, for many indications about the Lyapunov exponents and the singular vectors.

\vfill
\end{document}